\documentclass[11pt]{amsart}
\usepackage{}
\usepackage[T1]{fontenc}
\usepackage{lmodern}
\usepackage{ifthen}
\usepackage{amsfonts}
\usepackage{amsxtra}
\usepackage{amssymb}
\usepackage{amsthm}
\usepackage{array}
\usepackage[margin=1in]{geometry}
\usepackage{xcolor}
\definecolor{cite}{rgb}{0.50,0.00,1.00}
\definecolor{url}{rgb}{0.00,0.50,0.75}
\definecolor{link}{rgb}{0.00,0.00,0.50}
\usepackage[colorlinks,linkcolor=link,urlcolor=url,citecolor=cite,pagebackref,breaklinks]{hyperref}
\usepackage{mathtools}
\usepackage{mathrsfs}
\usepackage[all]{xy}
\usepackage[lite,abbrev,msc-links]{amsrefs}

\makeindex

\theoremstyle{definition} 
\newtheorem{Unity}{Unity}[section] 
\newtheorem*{defn*}{Definition} 
\newtheorem{defn}[Unity]{Definition} 

\theoremstyle{plain} 
\newtheorem*{thm*}{Theorem}
\newtheorem{thm}[Unity]{Theorem}
\newtheorem{prop}[Unity]{Proposition}
\newtheorem*{cor*}{Corollary}
\newtheorem{cor}[Unity]{Corollary}
\newtheorem{lem}[Unity]{Lemma}

\theoremstyle{remark} 
\newtheorem*{rmk*}{Remark}
\newtheorem{rmk}[Unity]{Remark}
\newtheorem{exmp}[Unity]{Example}

\numberwithin{Unity}{section}

\begin{document}
\title{The $d$-very ampleness of adjoint line bundles on quasi-elliptic surfaces}
\author[Yongming Zhang]{Yongming Zhang}
\email{zhangym@cumt.edu.cn}
\address{School of Mathematics, China University of Mining and Technology, Xuzhou, 221116, P. R. of China}
\maketitle
\begin{abstract}
In this paper, we give a numerical criterion of Reider-type for the $d$-very ampleness of the adjoint line bundles on quasi-elliptic surfaces, and meanwhile we give a new proof of the vanishing theorem on quasi-elliptic surfaces emailed from Langer and show that it is the optimal version.
\end{abstract}
\section{Introduction}

Let $X$ be a projective algebraic variety defined over an algebraically closed field $k$. Let $Z$ be a $0$-dimensional subscheme of $X$ which is called $0$-cycle of $X$. For an integer $d\geq0$, a line bundle $L$ on $X$ is called \emph{$d$-very ample} if for any $0$-cycle $Z$ with length $ Z\leq d+1$, the restriction map $$\Gamma(X,L)\longrightarrow\Gamma(Z,L\mid_Z)$$ is surjective. Note that $0$-very ampleness and $1$-very ampleness is equivalent to being generated by global sections and being very ample respectively.

Let $X^{[d]}$ be the Hilbert scheme of points on $X$ of length $d$. If $L$ is $d$-very ample then the restriction map associates to every $0$-cycle $Z$ of length $d+1$ a subspace of $H^0 (X,L)^*$ of dimension $d+1$ and this is indeed a morphism $$\phi_d:X^{[d+1]}\rightarrow Grass(d+1, H^0(X,L)^*).$$
And it is proved that $\phi_d$ is an embedding if and only if $L$ is $d+1$-very ample (see \cite{CG90}). Thus the $d$-very ampleness is geometrically a natural generalization of the usual notation of very ampleness.

Using Reider's method (\cite{Reider1988}), Beltrametti and Sommese  obtained a useful numerical criterion for the $d$-very ampleness of the adjoint line bundles in the case of surfaces in characteristic zero.
\begin{thm}[\cite{BS91}]\label{thm_BS}
Let $L$ be a nef line bundle on a complex smooth projective surface $X$ and suppose that $(L^2)\geq4r+5$. Then either $\mathcal{O}_X(K_X+L)$ is $r$-very ample or there exists an effective divisor $D$ containing some $0$-dimensional scheme of length $\leq r+1$ along which $r$-very ampleness fails, such that a power of the line bundle $L-2D$ has sections and
$$(D\cdot L)-r-1\leq (D^2)<\frac{1}{2}(D\cdot L)<r+1.$$
\end{thm}
In positive characteristic, by the results of N.I.Shepherd-Barron (\cite{SB91}), Theorem \ref{thm_BS} also works directly on surfaces neither of general type nor quasi-elliptic of Kodaira dimension $1$, and for the surface of general type, T. Nakashima used N.I.Shepherd-Barron's results to obtain a numerical criterion for the $d$-very ampleness of the adjoint line bundles (\cite{NT93}). Then H. Terakawa (\cite{TH99}) improved it and collected all such results on surfaces together.
\begin{thm}[\cite{TH99}]
Let $X$ be a nonsingular projective surface defined over an algebraically closed field of characteristic $p>0$. Let $L$ be a nef line bundle on
$X$. Assume that
$l:= L^2-4d-5\geq0$
and one of the following situations holds:
\begin{enumerate}
  \item $X$ is not of general type and further not quasi-elliptic of Kodaira dimension $1$;
  \item $X$ is of general type with minimal model $X^{\prime}$, $p\geq3$ and $l>(K_{X^{\prime}}^2)$;
  \item $X$ is of general type with minimal model $X^{\prime}$, $p=2$ and $l>{\rm max}\{(K_{X^{\prime}}^2),(K_{X^{\prime}}^2)-3\chi(\mathcal {O}_X)+2\}$.
\end{enumerate}
Then either $\mathcal{O}_X(K_X+L)$ is $d$-very ample or there exists an effective divisor $D$ containing some $0$-dimensional scheme of length $\leq d+1$ along which $d$-very ampleness fails, such that  $L-2D$ is $\mathbb{Q}$-effective and
$$(D\cdot L)-d-1\leq (D^2)<\frac{1}{2}(D\cdot L)<d+1.$$
\end{thm}

The purpose of this note is to study the adjoint linear system on the remaining case that $X$ is a quasi-elliptic surface and at the same time we give a new proof of the vanishing theorem on it.
\begin{thm}[Theorem \ref{d-veryample}]
Let $X$ be a quasi-elliptic surface over an algebraically closed field $k$ of characteristic $p$, and $F$ be a general fibre of the quasi-elliptic fibration $f:X\rightarrow C$. Let $L$ be a nef and big divisor on $X$.
Assume that $$L^2>4(d+1)$$
for a nonnegative integer $d$ then we have the following descriptions.
\begin{enumerate}
  \item When $p=2$, we assume that $(L\cdot F)>2$ in addition. Then either $\mathcal{O}_X(K_X+L)$ is $d$-very ample or there exists an effective divisor $B$ containing a $0$-cycle $Z^{(2)}=F^{*}Z$, which is the Frobenius pull back of a $0$-cycle $Z$ of $\deg Z\leq  d+1$ where the $d$-very ampleness fails, such that $2L-B$ is $\mathbb{Q}$-effective and
      $$2(L\cdot B)-4\deg Z \leq (B^2)\leq (L\cdot B) \leq 4\deg Z;$$
  \item  When $p=3$, we assume that $(L\cdot F)>1$ in addition. Then either $\mathcal{O}_X(K_X+L)$ is $d$-very ample or there exists an effective divisor $B$ containing a $0$-cycle $Z^{(3)}=F^{*}Z$, which is the Frobenius pull back of a $0$-cycle $Z$ of $\deg Z\leq d+1$ where the $d$-very ampleness fails, such that $3L-2B$ is $\mathbb{Q}$-effective and
      $$3(L\cdot B)-9\deg Z \leq (B^2)\leq \frac{3(L\cdot B)}{2} \leq 9\deg Z.$$
\end{enumerate}
\end{thm}
\begin{thm}[Langer, Theorem \ref{vanishing theorem}]\label{vanishing}
Let $X$ be a quasi-elliptic surface over an algebraically closed field $k$ of characteristic $p$, and $F$ be a general fibre of the quasi-elliptic fibration $f:X\rightarrow C$. Let $L$ be a nef and big divisor on $X$.
\begin{itemize}
  \item If $ p=2$ and $(L\cdot F)>2$, then $H^1(X,L^{-1})=0$,
  \item If $ p=3$ and $(L\cdot F)>1$, then $H^1(X,L^{-1})=0$.
\end{itemize}
\end{thm}

In fact, there are already some results along the line of Theorem \ref{vanishing} in the literature: \cite[Corollary 4.1]{Zheng17} and
\cite[Corollary 7.4]{Langer16} which seem stronger than Theorem \ref{vanishing} but may be wrong when $p=2$ . And from the email of Langer recently, this result follows by  modifying a small error in the calculation of length of the torsion part of $\Omega_F$ in \cite[Corollary 7.4]{Langer16} or by the first part of the proof of \cite[Corollary 4.1]{Zheng17} which still holds while the remaining part may go wrong. Here we give a new proof of this result and construct some examples to show that it is optimal.

Initially, our main skill is inspired by the method in Propositon 4.3 of \cite{DG15}. However we find that there is a small error in its proof (the general fibre $C_i$ may be non-reduced!), and it doesn't affect the results in Propositon 4.3 of \cite{DG15}. But it may lead a little trouble in Proposition 3.1 of \cite{C19}, which is a key step in the proof of Fujita's Conjecture on quasi-elliptic surfaces. However, the referee suggests to us a simple method to estimate the upper bound of anti-canonical degree on an integral curve by using adjunction formula, and it help us avoid the original tedious proof. The paper is organized as follows. The second section contains some preliminary materials which will be used in the last two sections to prove the main results. And we present a new proof of the vanishing theorem in the third section and study $d$-very ampleness of adjoint line bundle in the last section.

\textbf{Notations:}
\begin{itemize}
  \item Through this paper, $k$ is an algebraically closed field of characteristic $p>0$ and all varieties are defined over $k$;
  \item $K_X$ is the canonical divisor of a smooth projective variety $X$.
  \item For line bundles or divisors $A$ and $B$ on a projective surface, $(A\cdot B)$ denotes the intersection number of $A$ and $B$.
\end{itemize}

\textbf{Acknowledgement:} The author would like to thank Yi Gu and Chen Jiang for their useful discussions and suggestions and thank Adrian Langer for his warm letter. The author is supported by grant NSFC (No. 12101618). Finally the author would like to thank the referees for their carefully reading and useful contribution to Proposition \ref{keylem}.
\section{Preliminaries}
\subsection{The degree of anti-canonical divisor on curves}
The author thanks the refree for pointing out a stronger version of the following proposition with a simple proof. Indeed, as below the upper estimate of the anti-canonical degree on an integral part of the general fiber of a fibration over a curve could be easily obtained by the adjunction formula.
\begin{prop}\label{keylem}
Let $Y$ be a projective surface over an algebraically closed field. Let $F$ be an effective Cartier divisor on $Y$ such that $(F^2)=0$ and $h^0(\mathcal{O}(F))=1$,
and assume that $Y$ is
Gorenstein along $F$. Then $-(K_Y\cdot F)\leq 2$
is even and the equality holds if and only if $F$ is a smooth rational curve.
\end{prop}
\begin{proof}
By taking blow-ups away from $F$, we may assume that $Y$ is Gorenstein. Then the
adjunction formula says that $\omega_Y(F)|_F\simeq \omega_F$. By taking degree, we obtain
$(K_Y+F\cdot F)=\deg(\omega_Y(F)|_F)=\deg(\omega_F)=-2\chi(\mathcal{O}(F))=2h^1(\mathcal{O}(F))-2h^0(\mathcal{O}(F))$,
where we used the Riemann-Roch and the Serre duality on the Gorenstein curve $F$. Since $(F^2)=0$ and $h^0(\mathcal{O}(F))=1$,
$(K_Y\cdot F)=2h^1(\mathcal{O}(F))-2\geq-2$ and is even. Moreover the equality holds if and only if $p_a=h^1(\mathcal{O}(F))=0$.
\end{proof}
\begin{rmk}
The conditions in Proposition \ref{keylem} are satisfied when there is a fibration over a curve $f:Y\rightarrow C$ and $F$ is integeral and contained in a fibre of $f$.
\end{rmk}
\subsection{The Shepherd-Barron's result on the instability of locally free
sheaves}
\begin{defn}
A rank $2$ locally free sheaf $E$ on a smooth projective surface $X$ is \emph{unstable} if there is a short exact sequence
$$0\rightarrow\mathcal {O}(A)\rightarrow E\rightarrow I_Z\cdot\mathcal {O}(B)\rightarrow0$$
where $A, B\in {\rm Pic} (X)$, $I_Z$ is the ideal sheaf of an effective $0$-cycle $Z$ on $X$ and
$A - B \in C_{++}(X)$, the positive cone of $NS(X)$. (Recall that $C_{++}(X)=\{x\in NS(X)\mid x^2 >0$ and $x\cdot H>0$ for some ample divisor $H$ and hence every ample divisor $H$\}.) We say that $E$ is \emph{semi-stable} if it is not unstable.
\end{defn}

Let  $F:X\rightarrow X$ be the \emph{(absolute) Frobenius morphism} and $F^e$ be the  $e$-iteration of $F$. For any coherent sheaf $G$, and we write $G^{(p^e)}:=F^{e*}(G)$  simply. And for a morphism $f: Y\rightarrow X$, the \emph{relative $e$-iteration Frobenius morphism} $F_r^e$ is defined by the universal property the fibre product in following diagram
$$\xymatrix{
  Y \ar@/_/[ddr]_{f} \ar@/^/[drr]^{F^e}
    \ar@{.>}[dr]|-{F_r^e}                   \\
   & F^{e*}(Y) \ar[d] \ar[r]
                      & Y \ar[d]_{f}    \\
   & X \ar[r]^{F^e}     & X               .}
 $$
\begin{thm}[ Theorem 1, \cite{SB91}]\label{Bogomolov}
If $E$ is a rank $2$ locally free sheaf on a smooth projective surface X with $c_1^2(E)>4c_2(E)$, then there is a reduced and irreducible surface
$Y\subset \mathbb{P}=\mathbb{P}(E)$ such that
\begin{enumerate}
  \item the composite $\rho: Y \rightarrow X$ is purely inseparable, say of degree $p^n$;
  \item the $n$-iteration Frobenius $F^n: X \rightarrow X$ factors rationally through $Y$;
  \item putting $\widetilde{E} = F^{n*} E$, $\widetilde{\mathbb{P}}= \mathbb{P}(\widetilde{E})$ and letting $\psi: \widetilde{\mathbb{P}}\rightarrow \mathbb{P}$ be the natural map, we have $\psi^*Y = p^nX_1$, where $X_1$ is the quasi-section of $\widetilde{\mathbb{P}}$ corresponding to an exact sequence $0 \rightarrow \mathcal{O}(A) \rightarrow \widetilde{E} \rightarrow I_Z\cdot \mathcal{O}(B) \rightarrow 0$, where $A, B \in {\rm Pic}(X)$, $Z \in X$ is a $0$-cycle and $N=A-B$ lies in the positive cone $C_{++}(X)$ of $NS(X)$.
\end{enumerate}
\end{thm}
\begin{proof}
See Theorem 1 in \cite{SB91}.
\end{proof}
The $Y$ in the above theorem could be constructed as follows (which could also be considered as the image of $X_1$ under $\psi$ in Theorem 1 of \cite{SB91}):
Fix a semi-stable vector bundle $E$ of rank $2$ with $c_1^2(E)>4c_2(E)$. Suppose that $e>0$ is the smallest integer such that $F^{e*}E$ is unstable, then we have the following diagram, where $s$ is the section determined by the unstability of $F^{e*}E$, $Y=F^{e*}_r(X)$ is the pull back of this section under the relative $e$-iteration Frobenius $F_r^e$ which is reduced and irreducible, and  $\rho=\pi\!\!\mid_Y$ is a inseparable morphism  of degree $p^e$.
$$
\xymatrix{
  Y \ar@{_{(}->}[d] \ar[r]^{ \rho} & X \ar@{_{(}->}[d]^{s}\\
 \mathbb{P}(E) \ar[dr]_{\pi} \ar[r]^{ F_r^n} & \mathbb{P}(E^{(p^e)})\ar[d]^{F^{e*}(\pi)} \\
 & X  }
$$
\begin{prop}[\cite{SB91}, Corollary 5]\label{canonical divisor}
With the same assumption as Theorem \ref{Bogomolov},
$$K_Y\equiv \rho^*(K_X-\frac{p^e-1}{p^e}N).$$
\end{prop}

\subsection{The Tyurin's result on a construction of locally free sheaves}
 Let $X$ be a nonsingular projective surface defined over an algebraically closed field. Let $L$ be a line bundle on $X$. For a $0$-cycle $Z\in X^{[d]}$, consider the short exact sequence $0\rightarrow L\otimes I_Z\rightarrow L\rightarrow L\mid_Z\rightarrow0$,
then we have a long exact sequence
$$0\rightarrow H^0(X,L\otimes I_Z) \rightarrow H^0(X,L) \rightarrow H^0(L\mid_Z) \rightarrow H^1(X,L\otimes I_Z)  \rightarrow  H^1(X,L) \rightarrow  0.$$
Now put
$\delta(Z, L):=h^1(X,L\otimes I_Z)-h^1(X,L)$.
Note that the integer $\delta(Z, L)$ is nonnegative.
The cycle $Z$ is called\emph{ $L$-stable }(in the sense of Tyurin) if $\delta(Z, L) >\delta(Z_0, L)$
for any subcycle $Z_0$ of $Z$.
Note that $L$ is $d$-very ample if and only if $\delta(Z, L)=0$ for all $Z \in X^{[d+1]}$.
\begin{thm} [\cite{T87}, Lemma 1.2]\label{T87}
Let $L$ be a line bundle on a
nonsingular projective surface $X$ defined over an algebraically closed field
and let $Z$ be an $L$-stable $0$-cycle of $X$. Then there exists an extension
$$0 \rightarrow H^1(X,L\otimes I_Z) \otimes \mathcal{O}(K_X)\rightarrow E(Z, L)\rightarrow  L\otimes I_Z \rightarrow 0,$$
where $E(Z, L)$ is a locally free sheaf on $X$ of rank $h^1(X,L\otimes I_Z)+ 1$.
\end{thm}
\section{Vanishing theorem on quasi-elliptic surfaces}
In this section we give a new proof of the following vanishing theorem from a letter of Langer, which can also be obtained by a modification of a small error in the proof of \cite[Corollary 7.4]{Langer16} or by the first part of the proof of \cite[Corollary 4.1]{Zheng17} (the other part may go wrong), and then show it is optimal by some examples.
\begin{thm}[Langer]\label{vanishing theorem}
Let $X$ be a quasi-elliptic surface, and $F$ be a general fibre of the quasi-elliptic fibration $f:X\rightarrow C$. Let $L$ be a nef and big divisor on $X$.
\begin{itemize}
  \item If $ p=2$ and $(L\cdot F)>2$, then $H^1(X,L^{-1})=0$ and
  \item If $ p=3$ and $(L\cdot F)>1$, then $H^1(X,L^{-1})=0$.
\end{itemize}
\end{thm}
\begin{proof}
Assume that $H^1(X,L^{-1})\neq0$.
Let us take any nonzero element $0\neq\alpha\in H^1(X,L^{-1})$ which gives a non-split extension
$$0\rightarrow \mathcal {O}_X\rightarrow E\rightarrow L\rightarrow0.$$
 Corollary 17 in \cite{SB91} implies that $F^{e*}E$ will split when $e\gg 0$. Suppose that $e>0$ is the smallest integer such that it splits, then we have the following diagram, where $s$ is the section determined by the splitting of $F^{e*}E$, $Y=F^{e*}_r(X)$ is the pull back of this section which is reduced and irreducible and  $\rho=\pi\!\!\mid_Y$ is an inseparable morphism  of degree $p^e$.

$$
\xymatrix{
  Y \ar@{_{(}->}[d] \ar[r]^{ \rho} & X \ar@{_{(}->}[d]^{s}\\
 \mathbb{P}(E) \ar[dr]_{\pi} \ar[r]^{ F_r^n} & \mathbb{P}(E^{(p^e)})\ar[d]^{F^{e*}(\pi)} \\
 & X  }
$$

But the general fibre of $g= f \circ \rho: Y\rightarrow C$ may be not reduced: $$\rho^*(F)=p^{e-e_0}\widetilde{F}$$ with $ 0\leq e_0\leq e$ and $\widetilde{F}$ integral.

So by Proposition \ref{canonical divisor} we have
\begin{eqnarray*}-(K_Y\cdot \widetilde{F})&=&(\rho^*((p^e-1)L-K_X)\cdot \widetilde{F})\\
&=&p^{e_0-e}(\rho^*((p^e-1)L-K_X)\cdot \rho^*F)\\
&=&p^{e_0}((p^e-1)L-K_X\cdot F)\\
&=&p^{e_0}(p^e-1)(L\cdot F),
\end{eqnarray*}
where the last equality is due to that $f$ is a quasi-elliptic fibration.
By Proposition \ref{keylem},
we have the inequality $0<p^{e_0}(p^e-1)(L\cdot F)=-(K_Y\cdot \widetilde{F})\leq 2$ and $-(K_Y\cdot \widetilde{F})$ is even, so $p^{e_0}(p^e-1)(L\cdot F)=-(K_Y\cdot \widetilde{F})=2$, which gives that\\
when $p=2$,
\begin{itemize}
  \item $e=1$, $e_0=0$ and $(L\cdot F)=2$ or
  \item $e=1$, $e_0=1$ and $(L\cdot F)=1$
\end{itemize}
and when $p=3$,
\begin{itemize}
  \item $e=1$, $e_0=0$ and $(L\cdot F)=1$
\end{itemize}
So we get our result.
\end{proof}
\begin{cor}\label{thm}
Let $X$ be a quasi-elliptic surface and  $L$ a nef and big divisor on $X$, then we have
\begin{itemize}
  \item if $ p=2$ and $n\geq3$, then $H^1(X,L^{-n})=0$ and
  \item if $ p=3$ and $n\geq2$, then $H^1(X,L^{-n})=0$.
\end{itemize}
\end{cor}
Next, we will construct a quasi-elliptic surface to show that the above vanishing theorem is optimal.
\begin{exmp}\label{example}
Let $k$ be an algebraically closed field with ${\rm char}(k)=2$ and $C\subseteq \mathbb{P}_k^2=\mathrm{Proj}(k[X,Y,Z])$ be the plane curve defined by the equation:
$$Y^{2e}-X^{2e-1}Y=XZ^{2e-1},$$
where $e>1$ is a free variable. It is easy to check  that $C$ is a smooth curve and  $2g(C)-2=2e(2e-3)$.
Take $\infty:=[0,0,1]$ on $C$. Then $U:=C\backslash \infty=C\cap \{ X\neq 0\}$ is an affine open subset defined by $y^{2e}-y=z^{2e-1}$ with $y=Y/X$ and $z=Z/X$. As a result, $\mathrm{d}z$ is a generator of $\Omega^1_{C}|_{U}$ since $\mathrm{d}y=z^{2e-2}\mathrm{d}z$. So we have
$$
K_C=\mathrm{div}(\mathrm{d}z)=(2g(C)-2)\infty=2e(2e-3)\cdot \infty.
$$
Let $D=e(2e-3)\cdot \infty$, then  $C$ is a Tango curve with a Tango structure $L=\mathcal {O}_C(D)$ by \cite[Definition 2.1]{Zheng17} (see \cite[]{Tango72,Mu13,GZZ,Zh22} for more details about this example of Tango curves).

Set $e=3e_1$ and $N=e_1(2e-3)\cdot \infty$, then $L=\mathcal {O}_C (3N)$. Next we follow the same argument as section 2 in \cite{Zheng17} to construct a Raynaud surface $X$ which is an $l$-cycle cover of a ruled surface $P$ over $C$ with $l=p+1=3$
$$\phi: X\stackrel{\psi}{\longrightarrow} P\stackrel{\pi}{\longrightarrow} C.$$
Note that it is a quasi-elliptic fibration by \cite[Proposition 2.3]{Zheng17} and denote general fibre by $F$.

In this case, let $a=2$ and $b=1$, then $Z_{a,b}=\mathcal {O}_X(a\widetilde{E})\otimes \phi^*N^b=\mathcal {O}_X(2\widetilde{E})\otimes \phi^*N$ is ample and satisfies the condition in \cite[Theorem 3.7]{Zheng17}.
Hence we get $$H^1(X,Z_{2,1}^{-1})\neq0.$$

But $$(Z_{2,1}\cdot F)=2>1,$$ which leads to a contradiction with \cite[Corollary 7.4]{Langer16}.

Moveover, if we set $e_1=2e_2$ for some positive integer $e_2$, then $Z_{2,1}=2A$ with $A=\mathcal {O}_X(\widetilde{E})\otimes \mathcal {O}(e_2(2e-3)\cdot \infty)$ ample and $H^1(X,A^{-2})\neq0$, which leads to a contradiction with \cite[Corollary 4.1]{Zheng17}.
\end{exmp}
By a similar argument we could also obtain a quasi-elliptic surface $X$ in characteristic $3$ such that $H^1(X,A^{-1})\neq0$ with $A$ ample on $X$.
\section{Adjoint linear system on quasi-elliptic surfaces}
In this section we prove a theorem of $d$-very ampleness of Reider-type in positive characteristic on quasi-elliptic surfaces.
\begin{thm}\label{d-veryample}
Let $X$ be a quasi-elliptic surface, and $F$ be a general fibre of the quasi-elliptic fibration $f:X\rightarrow C$. Let $L$ be a nef and big divisor on $X$.
Assume that $$(L^2)>4(d+1)$$
for a nonnegative integer $d$, then we have the following descriptions.
\begin{enumerate}
  \item When $p=2$, we assume that $(L\cdot F)>2$ in addition. Then either $\mathcal{O}_X(K_X+L)$ is $d$-very ample or there exists an effective divisor $B$ containing a $0$-cycle $Z^{(2)}=F^{*}Z$, which is the Frobenius pull back of a $0$-cycle $Z$ of $\deg Z\leq d+1$ where the $d$-very ampleness fails, such that $2L-B$ is $\mathbb{Q}$-effective and
      $$2(L\cdot B)-4\deg Z \leq (B^2)\leq (L\cdot B) \leq 4\deg Z;$$
  \item When $p=3$, we assume that $(L\cdot F)>1$ in addition. Then either $\mathcal{O}_X(K_X+L)$ is $d$-very ample or there exists an effective divisor $B$ containing a $0$-cycle $Z^{(3)}=F^{*}Z$, which is the Frobenius pull back of a $0$-cycle $Z$ of $\deg Z\leq d+1$ where the $d$-very ampleness fails, such that $3L-2B$ is $\mathbb{Q}$-effective and
      $$3(L\cdot B)-9\deg Z \leq (B^2)\leq \frac{3(L\cdot B)}{2} \leq 9\deg Z.$$
\end{enumerate}
\end{thm}
\begin{proof}
Assume that $\mathcal{O}_X(K_X+L)$ is $(d-1)$-very ample but not $d$-very ample. Then there exist a $0$-cycle $Z$ of degree $d+1$ where the $d$-very ampleness fails.
By Theorem \ref{vanishing theorem}, we have $H^1(X,\mathcal{O}_X(K_X+L))=0$,
 then by Lemma \ref{rank2vector} we obtain a locally free sheaf $E$ of rank $2$ on $X$ sitting in the short exact sequence
$$0\rightarrow \mathcal {O}_X \rightarrow E \rightarrow I_Z\cdot L \rightarrow 0.$$
Moreover we have
$$c_1^2(E)-4c_2(E)=(L^2)-4 \deg Z= (L^2)-4(d+1)>0.$$
When $p=3$, by Lemma \ref{unstable} $F^*(E)$ is unstable, then we have the following diagram with exact vertical and horizontal sequences:
$$\xymatrix{
    &  & 0 \ar[d]_{} &  &  \\
    &  & \mathcal{O}(A) \ar[d]_{} \ar[rd]^{\sigma}&  &  \\
  0  \ar[r]^{} & \mathcal {O}_X\ar[r]^{} & E^{(3)} \ar[d]_{} \ar[r]^{} & I_Z^{(3)}\cdot L^{\otimes 3}  \ar[r]^{} & 0  \\
   &  & I_W\cdot \mathcal{O}(B) \ar[d]_{}  &  &  \\
  &  & 0 &  &   }
  $$
where $A, B \in {\rm Pic}(X)$ and $I_W$ is the ideal sheaf of a $0$-cycle $W$ on $X$ and
$A-B$ satisfies
\begin{itemize}
  \item $((A-B)^2)\geq 9(c_1^2(E)-4c_2(E))>0$ and
  \item $(A-B\cdot H)>0$ for any ample divisor $H$ on $X$.
\end{itemize}
Note that the composition map $\sigma:A\rightarrow E^{(3)}\rightarrow I_Z^{(3)}\cdot L^{\otimes 3}$ is nonzero, otherwise we have $A\hookrightarrow \mathcal {O}_X$ and $(A-B\cdot H)=(2A-3L\cdot H)<0$ for any ample divisor $H$ on $X$ which is a contradiction. Let $B=3L-A$ be the effective divisor defined by $\sigma$. Then it contains the $0$-cycle $Z^{(3)}=F^{*}Z$ since its local function is contained in the ideal sheaf $I_Z^{(3)}$ by the map $\sigma$. Thus we have
\begin{enumerate}
  \item $2(L\cdot B)\leq 3(L^2)$,
  \item $3(L\cdot B)-(B^2)\leq 9\deg Z$,
  \item $(L\cdot B)\geq0$ and
  \item $(L^2)(B^2)\leq (L\cdot B)^2$,
\end{enumerate}
where the first inequality is from the unstablity of $E^{(3)}$, the second one is obtained by computing the second Chern class of $E^{(3)}$ with the vertical and horizontal sequences and last one is form the Hodge index theorem. Put them together:
$$3(L\cdot B)-9\deg Z \leq (B^2) \leq \frac{(L\cdot B)^2}{L^2}\leq \frac{3(L\cdot B)}{2}.$$
So we have
$$3(L\cdot B)-9\deg Z \leq (B^2)\leq \frac{3(L\cdot B)}{2} \leq 9\deg Z.$$

When $p=2$, by Lemma \ref{unstable} we have $F^{*}(E)$ is unstable. Then by a similar argument as above, we will get an effective divisor $B$ containing the $0$-cycle $Z^{(2)}=F^{*}Z$ such that
$$2(L\cdot B)-4\deg Z \leq (B^2)\leq (L\cdot B) \leq 4\deg Z.$$
\end{proof}

\begin{lem}[\cite{BS91}, Lemma 1.2]\label{filtration}
Let $R$ be a Noetherian local ring and, $I$ and
$J$ be ideals of $R$ with $I \subseteq J$. Assume that ${\rm length}(R/I)<\infty$. Then there exists a chain
$$I=I_0 \subset I_1 \subset \cdots \subset I_r = J$$
of ideals of $R$ with ${\rm length}(I_i/I_{i-1})=1$ for $i=1,\ldots,r$.
\end{lem}

The following lemma is a slight improvement of Lemma 2.2 in \cite{NT93} by H. Terakawa. For reader's convenience, we present a proof here.
\begin{lem}[\cite{TH99}, Lemma 2.2]\label{rank2vector}
Let $X$ be a nonsingular projective surface defined over an algebraically closed field k and $L$ a line bundle on $X$ such that $H^1(X,\mathcal{O}_X(K_X+L))=0$.
Let $Z$ be a $0$-cycle with $\deg Z = d + 1$ where $d$ is a nonnegative integer. Assume that $\mathcal{O}_X(K_X+L)$ is $d-1$-very ample and the restriction map
$$\Gamma(\mathcal{O}_X(K_X+L))\rightarrow \Gamma(\mathcal {O}_X(K_X+L)|_Z)$$
is not surjective. Then there exists a rank $2$ locally free sheaf $E$ on $X$ which
is given by the short exact sequence
$$0\rightarrow \mathcal {O}_X \rightarrow E \rightarrow I_Z\cdot L \rightarrow 0,$$
where $I_Z$ is the ideal sheaf of $Z$.
\end{lem}
\begin{proof}
From the conditions we see that the cycle $Z$ is $\mathcal{O}_X(K_X+L)$-stable in the sense of Tyurin. Then by Theorem \ref{T87} we have a
locally free extension
$$0 \rightarrow H^1(X,(\mathcal{O}_X(K_X+L))\otimes I_Z) \otimes \mathcal {O}_X (K_X)\rightarrow E(Z, (\mathcal{O}_X(K_X+L)))\rightarrow  (\mathcal{O}_X(K_X+L))\otimes I_Z \rightarrow 0,$$
and it is sufficient to prove that $h^1((\mathcal{O}_X(K_X+L))\otimes I_Z)=1$.

By Lemma \ref{filtration}, we can take a sub-cycle $Z_0\subset Z$ of $\deg Z_0=d$. And we have the following diagram with exact rows.
$$\xymatrix{
  0 \ar[r] & (\mathcal{O}_X(K_X+L))\otimes I_{Z_0} \ar[r] & (\mathcal{O}_X(K_X+L)) \ar[r] & (\mathcal{O}_X(K_X+L))|_{Z_0} \ar[r] & 0  \\
  0 \ar[r] & (\mathcal{O}_X(K_X+L))\otimes I_{Z}  \ar[r]\ar@{_{(}->}[u]^i & (\mathcal{O}_X(K_X+L)) \ar[u]^{id}\ar[r] & (\mathcal{O}_X(K_X+L))|_{Z} \ar@{>>}[u]^j \ar[r] & 0
  }.$$
Note that ${\rm kernel}(j)={\rm coker}(i)=k$ and by Tyurin' stability, we have
 $$0=\delta(Z_0, \mathcal{O}_X(K_X+L))<\delta(Z, \mathcal{O}_X(K_X+L))={\rm dim \ coker}(H^0(\mathcal{O}_X(K_X+L))\rightarrow H^0((\mathcal{O}_X(K_X+L))|_{Z}))\leq 1.$$
Then considering the long exact sequence induced by the second row with the vanishing condition $H^1(\mathcal{O}_X(K_X+L))=0$, we obtain
$$H^1((\mathcal{O}_X(K_X+L))\otimes I_Z)={\rm coker}(H^0(\mathcal{O}_X(K_X+L))\rightarrow H^0((\mathcal{O}_X(K_X+L))|_{Z}))=k.$$
\end{proof}
There is a little error in the proof of \cite[ Propositon 4.3]{DG15} that $\rho^*F$ may be not reduced, and here we correct and improve it.
\begin{lem}\label{unstable}
Let $X$ be a quasi-elliptic surface over an algebraically closed field of characteristic $p$, and $F$ be a general fibre of the quasi-elliptic fibration $f:X\rightarrow C$. Let $E$ be a rank $2$ vector bundle on $X$ with $c_1^2(E)-4c_2(E)>0$ then $F^{*}E$ is unstable.
\end{lem}
\begin{proof}
We may assume $E$ is semi-stable. By Theorem \ref{Bogomolov}, let $e>0$ be the smallest integer such that $F^{e*}E$ is unstable:
$$0\rightarrow\mathcal {O}(A)\rightarrow F^{e*}E\rightarrow I_Z\cdot\mathcal {O}(B)\rightarrow0.$$

Considering the composition $g= f \circ \rho: Y\rightarrow C$,
$\rho^*F$ is a family of curves in $Y$ and the general fibre of $g$ may be not reduced: $$\rho^*(F)=p^{e-e_0}\widetilde{F}$$ with $ 0\leq e_0\leq e$ and $\widetilde{F}$ integral.

So by Propostion \ref{canonical divisor} we have
\begin{eqnarray*}-K_Y\cdot \widetilde{F}&=&(\rho^*(\frac{p^e-1}{p^e}(A-B)-K_X)\cdot \widetilde{F})\\
&=&p^{e_0-e}(\rho^*(\frac{p^e-1}{p^e}(A-B)-K_X)\cdot \rho^*F)\\
&=&p^{e_0}(\frac{p^e-1}{p^e}(A-B)-K_X\cdot F)\\
&=&(p^e-1)\frac{(A-B\cdot F)}{p^{e-e_0}}>0
\end{eqnarray*}
where the fourth equality is due to that $f$ is a quasi-elliptic fibration and the inequality is because $A-B$ is big.

By Proposition \ref{keylem},
we have  $$(p^e-1)\frac{(A-B\cdot F)}{p^{e-e_0}}=-(K_Y\cdot \widetilde{F})= 2,$$ which gives that
when $p=2$,
\begin{itemize}
  \item $e=1$, $e_0=0$ and $(A-B\cdot F)=4$ or
  \item $e=1$, $e_0=1$ and $(A-B\cdot F)=2$
\end{itemize}
and when $p=3$,
\begin{itemize}
  \item $e=1$, $e_0=0$ and $(A-B\cdot F)=3$ or
  \item $e=1$, $e_0=1$ and $(A-B\cdot F)=1$.
\end{itemize}
So we get our result.
\end{proof}


\begin{bibdiv}
\begin{biblist}

\bib{BS91}{article}{
   author={Beltrametti, M.},
   author={Sommese, A. J.},
   title={Zero cycles and $k$-th order embeddings of smooth projective
   surfaces},
   note={With an appendix by Lothar G\"{o}ttsche},
   conference={
      title={Problems in the theory of surfaces and their classification},
      address={Cortona},
      date={1988},
   },
   book={
      series={Sympos. Math., XXXII},
      publisher={Academic Press, London},
   },
   date={1991},
   pages={33--48},
   review={\MR{1273371}},
}
\bib{C19}{article}{
   author={Chen, Yen-An},
   title={Fujita's conjecture for quasi-elliptic surfaces},
   journal={Math. Nachr.},
   volume={294},
   date={2021},
   number={11},
   pages={2096--2104},
   issn={0025-584X},
   review={\MR{4371285}},
   doi={10.1002/mana.202000522},
}
\bib{CG90}{article}{
   author={Catanese, Fabrizio},
   author={G\oe ttsche, Lothar},
   title={$d$-very-ample line bundles and embeddings of Hilbert schemes of
   $0$-cycles},
   journal={Manuscripta Math.},
   volume={68},
   date={1990},
   number={3},
   pages={337--341},
   issn={0025-2611},
   review={\MR{1065935}},
   doi={10.1007/BF02568768},
}

\bib{DG15}{article}{
   author={Di Cerbo, Gabriele},
   author={Fanelli, Andrea},
   title={Effective Matsusaka's theorem for surfaces in characteristic $p$},
   journal={Algebra Number Theory},
   volume={9},
   date={2015},
   number={6},
   pages={1453--1475},
   issn={1937-0652},
   review={\MR{3397408}},
   doi={10.2140/ant.2015.9.1453},
}

\bib{GZZ}{article}{
   author={Gu, Yi},
   author={Zhang, Lei},
   author={Zhang, Yongming},
   title={Counterexamples to Fujita's conjecture on surfaces in positive
   characteristic},
   journal={Adv. Math.},
   volume={400},
   date={2022},
   pages={Paper No. 108271, 17},
   issn={0001-8708},
   review={\MR{4386546}},
   doi={10.1016/j.aim.2022.108271},
}
\bib{Langer16}{article}{
   author={Langer, Adrian},
   title={The Bogomolov-Miyaoka-Yau inequality for logarithmic surfaces in
   positive characteristic},
   journal={Duke Math. J.},
   volume={165},
   date={2016},
   number={14},
   pages={2737--2769},
   issn={0012-7094},
   review={\MR{3551772}},
   doi={10.1215/00127094-3627203},
}
\bib{Mu13}{article}{
   author={Mukai, Shigeru},
   title={Counterexamples to Kodaira's vanishing and Yau's inequality in
   positive characteristics},
   journal={Kyoto J. Math.},
   volume={53},
   date={2013},
   number={2},
   pages={515--532},
   issn={2156-2261},
   review={\MR{3079312}},
   doi={10.1215/21562261-2081279},
}
\bib{NT93}{article}{
   author={Nakashima, Tohru},
   title={On Reider's method for surfaces in positive characteristic},
   journal={J. Reine Angew. Math.},
   volume={438},
   date={1993},
   pages={175--185},
   issn={0075-4102},
   review={\MR{1215653}},
   doi={10.1515/crll.1993.438.175},
}
\bib{Reider1988}{article}{
   author={Reider, Igor},
   title={Vector bundles of rank $2$ and linear systems on algebraic
   surfaces},
   journal={Ann. of Math. (2)},
   volume={127},
   date={1988},
   number={no.~2},
   pages={309--316},
   issn={0003-486X},
   review={\MR{932299}},
}

\bib{SB91}{article}{
   author={Shepherd-Barron, N. I.},
   title={Unstable vector bundles and linear systems on surfaces in
   characteristic $p$},
   journal={Invent. Math.},
   volume={106},
   date={1991},
   number={2},
   pages={243--262},
   issn={0020-9910},
   review={\MR{1128214}},
   doi={10.1007/BF01243912},
}
\bib{Tango72}{article}{
   author={Tango, Hiroshi},
   title={On the behavior of extensions of vector bundles under the
   Frobenius map},
   journal={Nagoya Math. J.},
   volume={48},
   date={1972},
   pages={73--89},
   issn={0027-7630},
   review={\MR{314851}},
}

\bib{T87}{article}{
   author={Tyurin, A. N.},
   title={Cycles, curves and vector bundles on an algebraic surface},
   journal={Duke Math. J.},
   volume={54},
   date={1987},
   number={1},
   pages={1--26},
   issn={0012-7094},
   review={\MR{885772}},
   doi={10.1215/S0012-7094-87-05402-0},
}
\bib{TH99}{article}{
   author={Terakawa, Hiroyuki},
   title={The $d$-very ampleness on a projective surface in positive
   characteristic},
   journal={Pacific J. Math.},
   volume={187},
   date={1999},
   number={1},
   pages={187--199},
   issn={0030-8730},
   review={\MR{1674325}},
   doi={10.2140/pjm.1999.187.187},
}
\bib{Zh22}{article}{
   author={Zhang, Yongming},
   title={Strong non-vanishing of cohomologies and strong non-freeness of adjoint line bundles on $n$-Raynaud surfaces},
   journal={arXiv:2210.00967},
   date={2022},
}
\bib{Zheng17}{article}{
   author={Zheng, Xudong},
   title={Counterexamples of Kodaira vanishing for smooth surfaces of
   general type in positive characteristic},
   journal={J. Pure Appl. Algebra},
   volume={221},
   date={2017},
   number={10},
   pages={2431--2444},
   issn={0022-4049},
   review={\MR{3646309}},
   doi={10.1016/j.jpaa.2016.12.030},
}

\end{biblist}
\end{bibdiv}
\end{document}